\documentclass[12pt]{amsart}
\usepackage{amssymb}%
\usepackage{graphicx}%
\usepackage{amsthm}%
\usepackage{amsmath}%
\usepackage{amsfonts} 
\usepackage{latexsym}%
\usepackage{epsfig}%
\usepackage{epic}%
\usepackage{eufrak}%
\usepackage{amscd}%
\usepackage{color}%
\usepackage[latin1]{inputenc}
\usepackage[T1]{fontenc}
\usepackage[linktoc=all, pagebackref, hyperindex]{hyperref}%
\hypersetup{colorlinks,  citecolor=blue,  filecolor=blue,  linkcolor=red,  urlcolor=blue}
\theoremstyle{plain}
  \newtheorem{thm}{Theorem}[section]
  \newtheorem{prop}[thm]{Proposition}
  \newtheorem{lem}[thm]{Lemma}
  \newtheorem{cor}[thm]{Corollary}
  
\theoremstyle{definition}
  \newtheorem{dfn}[thm]{Definition}
  \newtheorem{ex}[thm]{Example}

\theoremstyle{remark}
  \newtheorem{rem}[thm]{Remark}
    {\begin{list}
        {\noindent\makebox[0mm][r]{\rm(\arabic{enumi})}}
        {\leftmargin=5.5ex \usecounter{enumi}}
    }
    {\end{list}}

\newcommand{\Max}{\operatorname{Max}\,}

\newcommand{\supp}{\operatorname{supp}\,}

\newcommand{\gr}{\operatorname{gr}}
\newcommand{\depth}{\operatorname{depth}\,}

\newcommand{\q}{\operatorname{q}\,}

\newcommand{\Min}{\operatorname{Min}\,}

\newcommand{\reg}{\operatorname{reg}}

\newcommand{\AP}{\operatorname{AP}}

\newcommand{\ord}{\operatorname{ord}}

\newcommand{\ld}{\operatorname{ld}}
\newcommand{\lm}{\operatorname{lm}}

\def\NN{{\mathbb N}}

\def\AA{{\mathbb A}}

\def\kk{\Bbbk}
\def\a{{\mathbf a}}
\def\b{{\mathbf b}}
\def\cc{{\mathbf c}}
\def\d{{\mathbf d}}

\def\p{{\mathbf p}}
\def\q{{\mathbf q}}
\def\z{{\mathbf z}}
\def\n{{\mathbf n}}

\def\s{{\mathbf s}}

\def\fm{{\mathfrak m}}
\def\fn{{\mathfrak n}}

\newcommand{\F}{\mathcal{F}}

\def\<{{\langle}}
\def\>{{\rangle}}

\newcommand{\excise}[1]{}
\begin{document}

\bibliographystyle{amsplain}
\title{ Homogeneous numerical semigroups  }

\author[Raheleh Jafari]{Raheleh Jafari$^1$}

\address{$^1$ Mosaheb Institute of Mathematics, Kharazmi Universtity, 50 Taleghani Avenue, 15618-36314  Tehran, Iran. }

\address{$^1$  School of Mathematics, Institute for
	Research in Fundamental Sciences (IPM) P. O. Box: 19395-5746,
	Tehran, Iran.}

\author[S. Zarzuela Armengou]{Santiago Zarzuela Armengou$^2$}

\address{$^2$ Departament de Matem\`{a}tiques i Inform\`{a}tica, Universitat de Barcelona, Gran Via 585, 08007 Barcelona, Spain.}

\address{$^1$Email: {\tt rjafari@ipm.ir}}

\address{$^2$Email: {\tt szarzuela@ub.edu}}

\thanks{Raheleh Jafari was supported in part by a grant from IPM (No. 94130129)}

\thanks{Santiago Zarzuela Armengou was supported by MTM2016-7881-P and 2017SGR-585}



\maketitle

\begin{abstract}
We introduce the concept of homogeneous numerical semigroups and  show that all homogeneous numerical semigroups with Cohen-Macaulay tangent cones are of homogeneous type. In embedding dimension three, we classify all numerical semigroups of homogeneous type into numerical semigroups with complete intersection tangent cones and the homogeneous ones which are not symmetric with Cohen-Macaulay tangent cones.
We also study the behavior of the homogeneous property by gluing and shiftings  to construct  large families of homogeneous numerical semigroups with Cohen-Macaulay tangent cones. In particular we show  that these properties  fulfill asymptotically in the shifting classes.
Several explicit examples are provided  along the  paper to illustrate the property.

\end{abstract}


\section{introduction}

Let $S$ be a numerical semigroup minimally generated by a sequence of positive integers $\n: 0 < n_1 < \cdots < n_d$. For any nonnegative integer $j$ one may consider the shifted sequence $\n + j: 0 < n_1+j < \dots < n_d+j$. Let $\kk$ be a field and $\kk[S]:=\kk[t^{n_1}, \dots , t^{n_d}]\subset \kk[t]$, where $t$ is a free variable, be the numerical semigroup ring defined by $S$. This is the homogeneous coordinate ring of the affine monomial curve in $\mathbb{A}_{\kk}^d$ defined parametrically by $x_1=t^{n_1}, \dots , x_d=t^{n_d}$. Denote by $I(\n)\subset \kk[x_1, \dots , x_d]$ the defining ideal of $\kk[S]$ obtained from the natural presentation $\kk[x_1, \dots ,x_d] \rightarrow \kk[S] \rightarrow 0$.
J. Herzog and H. Srinivasan conjectured that for $j\gg 0$ the Betti numbers of the ideals $I(\n +j)$ become periodic on $j$ with period $n_d-n_1$. In 2013, the conjecture was proven to be true by A.~V.~Jayanthan and H.~Srinivasan for $d=3$ \cite{JS}, by A.~Marzullo for some particular cases if $d=4$ \cite{Ma}, and by P.~Gimenez, I.~Sengupta and H.~Srinivasan in the case of arithmetic sequences \cite{GSS}. Finally, in 2014, T.~Vu gave a completely general positive answer in \cite{V}. One of the main ingredients of Vu's proof is that there exists a positive integer $N$ such that, for all $j>N$, any minimal binomial non-homogeneous generator of $I(\n +j)$ is of the form $x_1^{a_1}u - vx_d^{a_d}$, where $a_1, a_d$ are positive integers, $u, v$ are monomials in the variables $x_2, \dots , x_{d-1}$, and $\deg x_1^{a_1}u > \deg vx_d^{a_d}$. It is noteworthy that the bound $N$ depends on the Castelnuovo-Mumford regularity of $J(\n)$, the ideal generated by the homogeneous elements in $I(\n)$.
Let $I(\n)_{*}$ be the homogeneous ideal generated by the initial forms of the elements in $I(\n)$. Then, $\kk[x_1, \dots , x_d]/I(\n)_{*} \simeq G(S)$, the tangent cone of $\kk[S]$. By using this main ingredient in Vu's proof of the conjecture, J. Herzog and  D. I. Stamate proved in \cite{HS} that for any $j>N$, the Betti numbers of the ideals $I(\n +j)$ and $I(\n +j)_{*}$ coincide. Following the general definition given by J.~Herzog, M.~Rossi, and G.~Valla in \cite{HRV}, we say that $S$ if of homogeneous type if the above condition on the Betti numbers is satisfied. So the result of Herzog-Stamate may be rephrased by saying that $S+j$ is of homogeneous type for any $j>N$. Note that if a numerical semigroup $S$ is of homogeneous type then, the tangent cone $G(S)$ is Cohen-Macaulay.

\medskip

In this paper we introduce a new condition on $S$, to be homogeneous (cf. Definition~\ref{D}), that jointly with the Cohen-Macaulay property of $G(S)$ turns out to be equivalent to a property much similar to the one cited above as the main ingredient of Vu's proof of the Herzog-Srinivasan conjecture (cf. Theorem~\ref{**}). In fact, this property is given in terms of the Ap\'ery set of $S$ and so it can be checked in terms of the generating sequence of integers $\n$. We then show that if $S$ is homogeneous and the tangent cone $G(S)$ is Cohen-Macaulay, then $S$ is of homogeneous type (cf. Theorem~\ref{homogeneous type}). Taking into account that the Cohen-Macaulay property of $G(S)$ can also be checked in terms of the Ap\'ery set of $S$, we get a method to prove that a numerical semigroup $S$ is of homogeneous type that only depends on the Ap\'ery set of $S$ and, ultimately, on elementary computations on the sequence $\n$. In addition, we prove that there exists a positive integer $L$ such that for any $j>L$, all the numerical semigroups generated by sequences of the form $\n +j$ are homogeneous and have Cohen-Macaulay tangent cone (cf. Corollary~\ref{sh-h}), so in particular they are of homogeneous type. The novelty here is that the constant $L$ only depends on the sequence of integers $\n$ and can be easily computed. In fact, it can be computed in terms of what we call the shifting type of a numerical semigroup (cf. Definition~\ref{sh-t}): two numerical semigroups can be obtained one from another as a shifting if and only if they have the same shifting type. So our results say that in the class of numerical semigroups with the same shifting type, all numerical semigroups except a finite number, that only depends on its shifting type, are of homogeneous type.

\medskip

Homogeneous numerical semigroups recover those with a unique maximal expression introduced by J.C. Rosales in \cite{R1}. As a typical example of homogeneous numerical semigroups we have (among several others) those generated by generalized arithmetic sequences (cf. Example~\ref{arithmetic}). In this case, the tangent cone is also Cohen-Macaulay and so they are of homogeneous type. For some special cases of  the class of numerical semigroups generated by a generalized arithmetic sequence (namely, for $d\leq 4$ or $n_1 \leq 2d$), this property was proven by L.~Sharifan and R.~Zaare-Nahandi by completely different methods in \cite{SZ}. On the other hand, numerical semigroups of homogeneous type are not necessarily homogeneous. This is the case for some complete intersection numerical semigroups (cf.  Remark~\ref{koszul} and Example~\ref{example}). We then explore the difference between both classes of numerical semigroups, and found that in embedding dimension $3$ any numerical semigroup which is of homogeneous type has a complete intersection tangent cone or it is homogeneous with Cohen-Macaulay tangent cone (cf. Theorem~\ref{dim3}). In embedding dimension $4$ we give several examples illustrating this difference. In many cases we have checked we get the same conclusion as in embedding dimension $3$, so we could ask ourselves if the same is true for larger embedding dimensions, that is, if any numerical semigroup of embedding dimension $d\geq 4$ which is of homogeneous type, has a complete intersection tangent cone or it is homogeneous with Cohen-Macaulay tangent cone. But in fact this is not true as we show with a concrete example of embedding dimension $4$. By using gluing techniques we also show that for any embedding dimension $d$, there are infinitely many complete intersection numerical semigroups which are of homogeneous type but not homogeneous (cf. Corollary~\ref{d1}).

\medskip
Now, we briefly describe the content of the paper. All the necessary notation and machinery on numerical semigroups is introduced and fixed in Section 2. In Section 3 we prove our main results on homogeneous numerical semigroups, characterize them, and relate with the property of being of homogeneous type. In Section 4 we study in detail the case of embedding dimension 3 and provide different families of examples with embedding dimension 4. Then, in Section 5 we study the behavior of the homogeneous property by gluing, particularly for the case of extensions. Finally, in Section 6 we study shiftings and prove that the property of being homogeneous and having Cohen-Macaulay tangent cone fulfills asymptotically in the shifting classes. Many of the explicit examples along this paper have been computed by using the NumericalSgps package of GAP \cite{DGM}

\medskip

Part of this work has been developed during two stays that the first author has done in the Institute of Mathematics of the University of Barcelona (IMUB) in 2014 and 2016. Both authors would like to thank the IMUB for its hospitality and support. We also thank Anargyros Katsabekis for several discussions on the subject in the case of embedding dimension four, Dumitru Stamate for telling us about his results in \cite{St}, and Francesco Strazzanti for providing the example of a numerical semigroup of embedding dimension four which is of homogeneous type but neither  homogeneous nor  with a complete intersection tangent cone. Finally we would like to thank the anonymous referee for a careful reading of the manuscript and several useful comments and corrections to the paper.
\medskip

\section{preliminaries}
Let $\n:n_1,\ldots,n_d$ be a sequence of integers with $n_1<n_i$ for all $i=2,\ldots,d$, and  $S=\langle n_1,\ldots,n_d\rangle$ be the subsemigoup of $(\NN,+)$ generated by $\n$. We call $S$ a \emph{numerical semigroup} when $\gcd(n_1,\ldots,n_d)=1$ or equivalently $\NN\setminus S$ is a finite set (cf. \cite{GR}). Let $\kk$ be a field. Then, the sequence $\n$ gives rise to a \emph{monomial curve} $C:=C(\n)\subseteq \mathbb{A}_{\kk}^d$ whose parametrization is given by $x_1=t^{n_1},\ldots,x_d=t^{n_d}$. Let $\kk[S]:=\kk[t^{n_1},\ldots,t^{n_d}]\subset \kk[t]$ be the \emph{semigroup ring} generated by $S$ and set $P:=\kk[x_1,\ldots,x_d]$ the polynomial ring over $\kk$. Let $I(\n):=\ker(\varphi)$, where $\varphi:P\longrightarrow \kk[t]$ is the canonical homomorphism  defined by $\varphi(x_i)=t^{n_i}$. We have that $P/I(\n)\simeq \kk[S]$ and in the case that $S$ is a numerical semigroup, or $\kk$ is algebraically closed, then $\kk[S]$ is the coordinate ring of $C(\n)$ and so $I(\n)$ is in fact the defining ideal of $C(\n)$ (cf. \cite{RVZ}). Moreover, if $\n$ is a  minimal system of generators of $S$, the ideal $I(\n)$ only depends on $S$ and we set $I_S:=I(\n)$.
 Let $g:=\gcd(n_1,\ldots,n_d)$. The numerical semigroup assigned to $S$ is  defined as
\[N(S):=\langle n_1/g,\ldots,n_d/g \rangle.\]
Note that for any positive integer $s$, we have $s\in N(S)$ if and only if $gs\in S$.
\begin{rem}\label{iso}
The natural ring homomorphism \[\kk[N(S)]=\kk[t^{n_1/g},\ldots,t^{n_d/g}]\longrightarrow \kk[t^{n_1},\ldots,t^{n_d}]=\kk[S]\] is an isomorphism.
\end{rem}

Throughout this paper we consider the  natural grading on the polynomial ring.  For a vector  $\a=(a_1,\ldots,a_d)$ of non-negative integers, we use $x^\a$ to denote the monomial $x_1^{a_1}\ldots x_d^{a_d}$.  It is known that $I(\n)$ is generated by binomials $x^\a-x^{\b}$ where $\a$ and $\b$ are d-tuples of non-negative integers with  $\varphi(x^\a)=\varphi(x^\b)$. The set $\{x_i \mid a_i+b_i\neq 0\}$  of all variables which appear in $f=x^\a-x^\b$, is called the support of $f$ and is denoted by $\supp(f)$.
 By definition, each element $s\in S$ can be written as  $s=\sum^d_{i=1}a_in_i$ for some non-negative integers $a_i$. The vector $\a$ is called  a {\em{factorization}} of $s$ and  the set of all factorizations of $s$ is denoted by  $\mathcal{F}(s)$, which is obviously a finite set.
Let $|\a|=\sum^d_{i=1}a_i$ denote the total order of $\a$. Then the  maximum integer $n$ which is the total order of a vector  in $\F(s)$ is called the \emph{order} of $s$ and is denoted by
$\ord_S(s)$.  A vector $\a\in\F(s)$ with $|\a|=\ord_S(s)$, is called a \emph{maximal factorization} of $s$ and $s=\sum^d_{i=1}a_in_i$ is called a \emph{maximal expression} of $s$. For a vector $\a$ of non-negative integers, we set $s(\a)=\sum^d_{i=1}a_in_i$.

\begin{rem}
Let $s$ be an element of $S$ and $M=S\setminus\{0\}$ be the maximal ideal of $S$. Then the order of $s$ is  the maximum  integer $n$ such that $nM$
contains  $s$.  In other words, $s\in nM\setminus (n+1)M$ if and only
if $n=\ord_S(s)$.
\end{rem}

We use two partial orderings $\preceq$ and $\preceq_M$ on $S$ where, for all elements $x$ and $y$ in $S$, $x\preceq y$ if there is an element $z\in S$ such that $y=x+z$ and $x\preceq_My$ if $y=x+z$ with $\ord_S(y)=\ord_S(x)+\ord_S(z)$ for some $z\in S$. For a finite subset $T\subset S$, considering these orderings, the maximal elements of
$T$ are denoted respectively by $\Max T$ and
$\Max_M T$ and the minimal elements of $T$ are respectively denoted by $\Min T$ and $\Min_MT$. It is clear that $\Max T \subseteq \Max_MT$ and $\Min T\subseteq\Min_MT$.
The following easy fact will be used frequently in our approach.
\begin{lem} Let $s$ and $y$ be elements in $S$.
Then
\begin{enumerate}
\item  $s\preceq y$ if and only if there exist factorizations  $\a\in\F(s)$ and $\b\in\F(y)$ such that the monomial $x^\b$ is divisible by  $x^\a$.
\item $s\preceq_My$ if and only if there exist maximal factorizations $\a\in\F(s)$ and $\b\in\F(y)$ such that  the monomial $x^\b$ is divisible by  $x^\a$.
\end{enumerate}
\end{lem}

The \emph{leading term} of a non-zero element $f\in P$ is the homogeneous summand of $f$ with least degree, which we  denote  by $f_*$ and set $\ld(f)$ for the degree of $f_*$. For an ideal $I\subset P$  we set $I_*\subset P$ be the graded ideal generated by all $f_*$ with $f\in I$.
A set $\{f_1,\ldots,f_r\}$ of elements of $I$ is called a \emph{standard basis} for $I$ if $I_*$ is generated by $\{{f_1}_*,\ldots,{f_r}_*\}$.

In the sequel, we consider the coordinate ring $G(S)$ of the tangent cone of $R=\kk[S]$, which is precisely the  associated graded ring $\gr_\fm(R)$ of $R$ with respect to the maximal ideal $\fm=(t^{n_1},\ldots,t^{n_d})$. Note that
\begin{equation}
G(S)=\gr_\fm(R)\cong \kk[x_1,\ldots,x_d]/I(\n)_*
\end{equation}

\medskip
Let $A=\kk[\![x_1,\ldots,x_d]\!]$ be the formal power series ring and
$\kk[\![S]\!]=\kk[\![t^{n_1},\ldots,t^{n_d}]\!]$ be the local ring associated to $S$. In other words $\kk[\![S]\!]=\kk[\![x_1,\ldots,x_d]\!]/J(\n)$,  where
$J(\n)$ is the kernel of the canonical
homomorphism $\psi: A\longrightarrow \kk[\![t]\!]$ defined by $\psi(x_i)=t^{n_i}$. Note that  the one dimensional integral domain $k[\![S]\!]$ is  indeed the $\fm$-adic completion of $\kk[S]$.
The smallest element of $N(S)$, $e:=n_1/g$, is equal to the multiplicity of $\kk[\![N(S)]\!]\cong \kk[\![S]\!]$. Indeed $t^{e}$ generates a minimal reduction of the maximal ideal of $\kk[\![N(S)]\!]$. The element $e$ is called the \emph{ multiplicity} of $S$ and is denoted by $m(S)$.

\medskip

For an element $f\in \kk[S]$ with $\fm$-adic order $d$, the residue class of $f$ in $\fm^d/\fm^{d+1}$ is called the \emph{initial form} of $f$ and is denoted by $f^*$.
Note that   $(t^s)^*(t^{s'})^*=0$ for two elements $s,s'\in S$, if
and only if $\ord_S(s+s')>\ord_S(s)+\ord_S(s')$. Hence if we set
\[T(S) :=\{s\in S;\, \exists \,c >0 \textrm{ with }  \ord(s+ c n_1)> \ord(s) +c  \},
\]
then $T(S)$ is precisely the set of elements of $s\in S$ such that $(t^s)^*$ annihilates some power of $t^{n_1}$. These elements are  called \emph{torsion} elements of $S$.  In the case that $S$ is a numerical semigroup, it is known that $G(S)$ is Cohen-Macaulay if and only if $T(S)=\emptyset$, which is  proved for the first time by A. Garc\'ia in
\cite[Theorem 7, Remark 8]{G}, see also \cite[Remark 2.11]{CJZ} for
a different proof. Now, by the isomorphism in Remark~\ref{iso}, we have the following statement.

\begin{prop}\label{CM}
	$G(S)$ is Cohen-Macaulay if and only if $T(S)=\emptyset$.
\end{prop}

Let $<$ be a monomial order on  $P$ and let $f=\sum_{i=1}^n r_ix^{\a_i}$ be a non-zero polynomial with $r_i\in \kk$.  The \emph{leading monomial} of $f$ with respect to $<$, denoted by $\lm_<(f)$, is the biggest monomial with respect to $<$ among the monomials $\{x^{\a_1},\ldots,x^{\a_n}\}$.
A  set of polynomials
$G=\{f_1,\ldots,f_n\}$ of an ideal $I$ is called a \emph{Gr\"obner basis} of $I$ with respect to $<$, if $\{\lm_<(f_1),\ldots,\lm_<(f_n)\}$ is the set of generators for the monomial ideal $\lm_<(I)=(\lm_<(f) \mid f\in I)$. Since the monomial ideal $\lm_<(I)$ has a unique minimal set of monomial generators, the set of leading monomials of elements of any minimal Gr\"obner basis for  $I$, is a unique set.
A Gr\"obner basis $G$ is called \emph{reduced} if the coefficient of $\lm_<(f_i)$ in $f_i$ is one for all $1\leq i\leq n$ and for $i\neq j$, none of the monomials of $\supp(f_j)$ is divisible by $\lm_<(f_i)$. Any Gr\"obner basis of $I$ is a generating set for $I$, a reduced Gr\"obner basis exists and it is uniquely determined (cf. \cite[Theorem~2.2.7]{HH}).

\begin{rem}\label{Gb}
	Let $G=\{f_1,\ldots,f_r\}$ be the reduced Gr\"obner basis of $I$. For each $f\in I$,  $\lm_<(f)$ is divisible by $\lm_<(f_i)$ for some $i=1,\ldots r$.  Therefore $\lm_<(f_i)$ does not divide $\lm_<(f_j)$ for all $j\neq i$.
\end{rem}

\begin{rem}
	Since $I(\n)$ is a binomial ideal, the  reduced Gr\"obner basis of $I$ consists of binomials by \cite[Proposition~2.3.7]{HH}.	
\end{rem}

 We  consider  the negative degree reverse lexicographical ordering with $x_2>\cdots>x_d>x_1$. We denote this local term order by $<_{ds}$, i.e.
   $x^\b<_{ds}x^\a$ precisely when one of the following statements holds:
\begin{itemize}\item $|\b|>|\a|$; or
	\item $|\b|=|\a|$ , $b_1>a_1$; or
\item $|\b|=|\a|$ , $b_1=a_1, b_d=a_d,\ldots,b_{i+1}=a_{i+1}, b_i>a_i$  for some $2\leq i\leq d$.\end{itemize}

\medskip
\begin{rem}\label{hds}
	Let $f=\sum^n_{i=1}r_ix^{\a_i}$ be a homogeneous polynomial, where $r_i\in \kk$. Let $x^{\a_1}=\lm_{<_{ds}}(f)$ and  $x_1\in\supp(x^{\a_1})$. Since $x^{\a_i}<_{ds}x^{\a_1}$ and $|\a_i|=|\a_1|$, we have  $x_1\in\supp(x^{\a_i})$ for all $i=2,\ldots,d$. In particular $x_1$ divides $f$.
\end{rem}

\begin{prop}[\text{cf. \cite[Lemma 2.7]{AMS}}]\label{ds}
Let $S$ be a numerical semigroup. Let $G=\{f_1,\ldots,f_s\}$ be a minimal Gr\"obner basis of $I(\n)$
with respect to $<_{ds}$. Then
$G(S)$ is Cohen-Macaulay if and only if $x_1$ does not divide $\lm_{<_{ds}}(f_i)$ for $1\leq i\leq d$.
\end{prop}

\begin{cor}\label{Codc}
Let $G$ be a minimal Gr\"obner basis of $I(\n)$ with respect to $<_{ds}$. Then $x_1$ does not divide $\lm_{<_{ds}}(f)$ for all $f\in G$, if and only if $x_1$ does not divide the leading monomials of elements of any minimal  Gr\"obner basis of $I(\n)$ with respect to $<_{ds}$.
\end{cor}

\section{Semigroups with homogeneous Ap\'ery sets}

 For an element $s\in S$, the Ap\'ery set of $S$ with respect to $s$ is defined as
\[\AP(S,s)=\{x\in S \mid x-s\notin S\}.\]
Let $g:=\gcd(n_1,\ldots,n_d)$ and $N(S)=\langle n_1/g,\ldots,n_d/g \rangle$ be the numerical semigroup assigned to $S$. Then for any positive integer $s$, we have $s\in N(S)$ if and only if $gs\in S$. Hence
\[\AP(S,s)=\{gx \mid x\in\AP(N(S),s/g)\}.\]
The Ap\'ery set $\AP(N(S),t)$, for  $t\in N(S)$,  is indeed the set of the smallest elements in $N(S)$ in each congruence class modulo $t$ and has $t$ elements.

\medskip

 Given $0\neq s\in S$,  the set of lengths of $s$ in $S$ is defined as
 \[\mathcal{L}(s)=\{\sum_{i=1}^dr_i \mid s=\sum^d_{i=1}r_in_i , r_i\geq 0\}.\]

 \begin{dfn}\label{D}
 	A subset $T\subset S$ is called \emph{homogeneous} if either it is empty or $\mathcal{L}(s)$ is singleton for all $0\neq s\in T$. In other words, all expressions of elements in $T$ are maximal.   The numerical semigroup $S$ is called homogeneous, when the Ap\'ery set  $\AP(S,n_1)$ is homogeneous.
 	 \end{dfn}

\begin{ex}
Let $d=2$. Then $\AP(S,n_1)=\{0,n_2,\ldots,(n_1-1)n_2\}$ is clearly homogeneous.
\end{ex}

\begin{ex}\label{maximal embedding}
 A numerical semigroup  is called of (almost) maximal embedding dimension if its multiplicity is equal to the embedding dimension (minus one). As in this case, the Ap\'ery set is precisely the minimal set of generators (and one more element of order two), $S$ is homogeneous.
\end{ex}

\begin{ex}\label{F}
Let $b>a>3$ be coprime integers. The semigroup $H_{a,b}=\<a,b,ab-a-b\>$ is called Frobenius semigroup (cf.\cite{HS}), since it is obtained from the semigroup $\<a,b\>$ by adding its Frobenius number. Note that $\AP(H_{a,b},a)=\{0,b,\ldots,(a-2)b,ab-a-b\}$ and so it is homogeneous.

Note also that $(ab-a-b)+a=(a-1)b$ has order $a-1>2$ and so $G(H_{a,b})$ is not Cohen-Macaulay. Indeed $T(H_{a,b})=\{ab-a-b\}$.
\end{ex}

\begin{ex}\label{generic}
	Recall that $I_S$ is called \emph{generic} if it is generated by binomials with full support i.e. all variables belong to the support of these binomials (cf. \cite{PS}). Assume that $I_S$ is generic and let  $E$ be the set of  minimal set of generators with full support. Let $s\in\AP(S,n_i)$ with two expressions $s=\sum_{j\neq i}a_jn_j=\sum_{j\neq i}b_jn_j$. Then $x^\a-x^\b\in I_S$ should be generated by elements of $E$. But all elements in $E$ have $x_i$ in their support, which is not possible. Hence the elements of $\AP(S,n_i)$ have unique expressions, in particular  $\AP(S,n_i)$ is homogeneous.
\end{ex}

\begin{ex}\label{arithmetic}
 Let $S$ be the numerical semigroup minimally generated by the generalized arithmetic sequence $n_0$, $n_i=hn_0+it$ where $n_0$, $t$ and $h$ are given positive integers and $i=1,\ldots,d$. Since we assume that $S$ is a numerical semigroup, the minimal generators are relatively prime and, then, $\gcd(n_0,t)=1$. We know from \cite{S} (see also \cite{M}) that
\[\AP(S,n_0)=\{hn_0\lceil\frac{r}{d}\rceil+tr \ ;  \ 0\leq r<n_0\}.\]
Let $s\neq 0$ be an element of $\AP(S,n_0)$. Then, any expression of $s$ cannot involve the generator $n_0$, hence if we have an expression of $s$ of length $l\geq 1$ it must be of the form $s=lhn_0+tr$ with $r=\sum_{i=1}^l a_i$, where $1\leq a_i \leq d$ for any $i$. Let $r=qn_0 + r_1$ where $0\leq r_1 < n_0$. Then, $ld \geq r = qn_0 +r_1$ and so $l \geq \frac{qn_0}{d} + \frac{r_1}{d}$. Because $l$ is a positive integer this implies that $l \geq \lceil\frac{r_1}{d}\rceil$.
Let now be any element $a=lhn_0+tr$ of $S$ with $l\geq \lceil\frac{r_1}{d}\rceil$ and let $r=qn_0 + r_1$ where $0\leq r_1 < n_0$.
Then $w:=lhn_0+tr_1 \in S$ and consequently $a=w+ tqn_0$ does not belong to $\AP(S,n_0)$ if $r\geq r_0$.

Let $0 \neq s\in\AP(S,n_0)$ with two expressions of length $l$ and $l'$. Then
\[s=lhn_0+tr=l'hn_0+tr'.\]
Since $\gcd(n_0,t)=1$, we have $lh=l'h+\alpha t$ and $r'=r+\alpha n_0$ for some integer $\alpha$. On the other hand  $r,r'<n_0$, since $s$ belongs to  $\AP(S,n_0)$. Hence  $\alpha=0$ and consequently $l=l'$.
Therefore  $\AP(S,n_0)$ is a homogeneous set. Implicitly we have $\ord_S(hn_0\lceil\frac{r}{d}\rceil+tr)=\lceil\frac{r}{d}\rceil$, for all  $0\leq r<n_0$.
\end{ex}

\begin{lem}\label{key}
	Let $\cc$ be a factorization of an element $s$ of  $S$. Then there exists a minimal set of generators $E$ of $I_S$ such that   each binomial $x^\a-x^\b\in E$ with $s(\a)\notin\AP(S,s)$,  	
	has one term divisible by $x^\cc$.
\end{lem}
\begin{proof}
	Let $E_1$ be a finite set of generators for $I_S$. If $g=x^\a-x^\b\in E_1$ with $z=s(\a)\notin\AP(S,s)$, then  $z=\sum^d_{i=1}d_in_i$, where  $d_i\geq c_i$.
	Now \[E_2=(E_1\setminus\{g\})\cup\{x^\a-x^\d, x^\b-x^\d\}\] is again a finite set of generators for $I_S$. Continuing in this way, we get  a generating set with the desired property and removing extra elements we have a minimal set of generators.
\end{proof}	

\begin{lem}\label{lem}
Let $E$ be a subset of homogeneous binomials in $I_S$ and $J$ be the ideal generated by $E$.
Then any binomial in $J$  is homogeneous.
\end{lem}
\begin{proof}
Let $f=x^\a-x^\b$ be a binomial in $J$.
If $|\a|\neq|\b|$, then $x^\a$ is a homogeneous component of $f$. Hence $x^\a\in J\subseteq I_S$, a contradiction.
\end{proof}	

\begin{prop}\label{homog-prop}
	Let $s\in S$.  The following statements are equivalent.
	\begin{enumerate}
		\item $\AP(S,s)$ is homogeneous.
		\item For any factorization $\cc$ of $s$, there exists a minimal set of generators $E$ for $I_S$ such that one  term of  each non-homogeneous element of $E$ is  divisible by $x^\cc$.
\item There is a factorization $\cc$ of $s$, and  a minimal set of generators $E$ for $I_S$ such that one  term of  each non-homogeneous element of $E$ is  divisible by $x^\cc$. 
	\end{enumerate}	
\end{prop}
\begin{proof}
	(1)$\Rightarrow$(2): Let $\cc$ be a factorization of $s$. By Lemma~\ref{key}, there exists a minimal set of generators $E$ of $I_S$ such that  each binomial $x^\a-x^\b\in E$ with $s(\a)\notin\AP(S,s)$,   has one term divisible by $x^\cc$.
		Let $f=x^\a-x^\b$ be a non-homogeneous binomial in $E$. Then
	$x=\sum^d_{i=1}a_in_i=\sum^d_{i=1}b_in_i$ has two expressions with different lengths. Hence $x\notin\AP(S,s)$ and so $x^\a$ or $x^\b$ is divisible by $x^\cc$.
	
	\vspace{2mm}
	(2)$\Rightarrow$(3): It is clear.
	
		\vspace{2mm}
	(3)$\Rightarrow$(1):  Let 	
	 $\{f_1,\ldots,f_n\}=\{f\in E \mid \mbox{ no term of } f \mbox{ is divisible by } x^\cc \}$ and let $z\in\AP(S,s)$ with two expressions $z=\sum^d_{j=1}a_jn_j=\sum^d_{j=1}b_jn_j$. Then $x^\a-x^\b\in (f_1,\ldots,f_n)$. Since all $f_j$ are homogeneous by the hypothesis, we get  $|\a|=|\b|$, from Lemma~\ref{lem}.

\end{proof}

\begin{cor}\label{homog}
The following statements are equivalent for $i=1,\ldots,d$.
\begin{enumerate}
\item $\AP(S,n_i)$ is homogeneous.
\item There exists a minimal set of generators $E$ for $I_S$ such that $x_i$ belongs to the support of all non-homogeneous elements of $E$.
\end{enumerate}
\end{cor}

Consider the natural map
\[\pi:P = \kk[x_1,\ldots,x_d]\longrightarrow \bar{P}:=\kk[x_2,\ldots,x_d],\] where $\pi(x_1)=0$ and $\pi(x_i)=x_i$ for $i=2,\ldots,d$. Then
\[\bar{P}/\pi(I_S)\cong P/(I_S,x_1).\]

For a polynomial $f\in P$, we set $\bar{f}:=\pi(f)$  and for a vector of non-negative integers $\a=(a_1,\ldots,a_d)$, we set
\[\bar{\a}=\left\{\begin{array}{ll} \a & \mbox{ if } a_1=0 \\ (a_1-1,a_2,\ldots,a_d) & \mbox{ if } a_1>0.  \end{array}\right.\]


\medskip
\begin{rem}\label{maximal-bar}
Let $E$ be  a minimal set of generators for $I_S$ and $x^\a-x^\b\in E$. Note that $a_ib_i=0$ for all $i=1,\ldots,d$, since $I_S$ is prime and $E$ is a minimal set of generators. If $b_1\neq 0$, then we  replace $x^\a-x^\b$ by $x^\a-x_1x^{\cc}$ and $x^\b-x_1x^{\cc}$, where $\cc$ is a maximal expression in $\F(s(\bar{b}))$.
Finally, we will have a set of minimal generators with the property that, for any $x^\a-x^\b\in E$ with $b_1\neq 0$, we have $\bar{\b}$ is a maximal expression.
\end{rem}

Next theorem is one of the main results in the paper.

\begin{thm}\label{**}
Let $S$ be a numerical semigroup. The following statements are equivalent.
\begin{enumerate}
\item  $S$ is homogeneous and $G(S)$ is Cohen-Macaulay.
\item  For all $x^\a-x^\b\in I_S$ with $|\a|>|\b|$, we have $s(\a)\notin\AP(S,n_1)$. Moreover if $\bar{\b}$ is a maximal factorization, then  $a_1\geq b_1$.
\item  There exists a minimal set of binomial generators $E$ for $I_S$ such that for all $x^\a-x^\b\in E$ with $|\a|>|\b|$, we have   $a_1\neq 0$.
\item There exists a minimal set of binomial generators $E$ for $I_S$ which is a standard basis and for all $x^\a-x^\b\in E$ with $|\a|>|\b|$, we have   $a_1\neq 0$.
\item There exists a minimal Gr\"obner  basis $G$ for $I_S$ with respect to $<_{ds}$, such that  $x_1$ belongs to the support of all non-homogeneous elements  of $G$ and $x_1$ does~not divide  $\lm_{<_{ds}}(f)$,  for all  $f\in G$.
\end{enumerate}
\end{thm}
\begin{proof}
(1)$\Rightarrow$(2):
The first statement follows by Definition~\ref{D}.  For the second part,
let $s:=s(\bar{\b})$. Then $\ord_S(s)=\sum^d_{i=1}\bar{b}_i$. If $a_1<b_1$, then $\bar{b}_1=b_1-1$ and $s+n_1=\sum^d_{i=1}a_in_i$. Hence $\ord_S(s+n_1)\geq\sum^d_{i=1}a_i>1+\ord_S(s)$, which implies that $s$ is a torsion element and this contradicts  Proposition~\ref{CM}.

\vspace{2mm}
(2)$\Rightarrow$(3):
Let $E_1$ be a finite set of generators for $I_S$.
Let $g=x^\a-x^\b\in E_1$ with $|\a|>|\b|$ and $a_1=0$. Then $s=\sum^d_{i=1}a_in_i=\sum^d_{i=1}b_in_i\notin\AP(S,n_1)$ by the statement~(2). Therefore $s=\sum^d_{i=1}c_in_i$, where  $c_1>0$ and  $\sum^d_{i=2}c_in_i$ is a  maximal expression.  Now, $E_2=(E_1\setminus\{g\})\cup\{x^\a-x^\cc, x^\b-x^\cc\}$ is again a finite set of generators for $I_S$.
Note also that $c_1>a_1$ and $c_1>b_1$ by the statement of~(2). Continuing in this way, we get  a generating set with the desired property and removing extra elements we have a minimal set of generators.

\vspace{2mm}
(3)$\Rightarrow$(4): We use the idea in the proof of \cite[Theorem~1.4]{HS}.
Let \[E=\{f_1,\ldots,f_t,g_1,\ldots,g_r\},\] where $f_1,\ldots,f_t$ are  homogeneous binomials and $g_1,\ldots,g_r$ are non-homogeneous.
Let $g_i=x^{\a_i}-x^{\b_i}$ with $|\a_i|>|\b_i|$. Then $x_1$ divides $x^{\a_i}$ and so
 $\pi(g_i)=x^{\b_i}$.  Hence
 $\pi(I_S)$ is generated by $B=\{\bar{f_1},\ldots,\bar{f_t},x^{\b_1},\ldots,x^{\b_r}\}$. Note that  $\bar{f_i}$ is either  equal to $f_i$ which is homogeneous, or it is a monomial. Therefore $B$ is a homogeneous set of generators for $\pi(I_S)$. In particular $B$ is a standard basis of $\pi(I_S)$. Now, using \cite[Lemma~1.2]{HS}  we get that $E$ is a standard basis of $I_S$ (see the proof of \cite[Theorem~1.4]{HS}).

\vspace{2mm}
(4)$\Rightarrow$(5):
Let $G=\{g_1=x^{\cc_1}-x^{\d_1},\ldots,g_s=x^{\cc_s}-x^{\d_s}\}$ be the reduced  Gr\"obner  basis with
$\lm_{<_{ds}}(g_i)=x^{\cc_i}$ and $E=\{f_1=x^{\a_1}-x^{\b_1},\ldots,f_r=x^{\a_r}-x^{\b_r}\}$. If $g_i$ is non-homogeneous, then $|\cc_i|<|\d_i|$ and so ${g_i}_*=x^{\cc_i}$. Since $E$ is a standard basis, ${f_j}_*$ divides $x^{\cc_i}$ for some $j$. Therefore $f_j$ is non-homogeneous, say  $|\b_j|<|\a_j|$. Then ${f_j}_*=x^{\b_j}=\lm_{<_{ds}}(f_j)$ divides $\lm_{<_{ds}}(g_i)$, which implies that $x^{\b_j}=\lm_{<_{ds}}(g_i)=x^{\cc_i}$ by Remark~\ref{Gb}. Note that $x_1$ is not in the support of homogeneous elements of $G$ from Remark~\ref{hds}. Now replacing $g_i$ with $f_j$, we get the desired Gr\"obner basis.

 \vspace{2mm}
 (5)$\Rightarrow$(1):
 From Corollary~\ref{homog} it follows that $S$ is homogeneous and Proposition~\ref{ds} implies the Cohen-Macaulayness of $G(S)$.
\end{proof}

The following example illustrates the fact that even if $S$ is homogeneous and $G(S)$ is Cohen-Macaulay, not any minimal generating set for $I_S$ satisfies the properties of the theorem.

\begin{ex}
	Let $S=\langle 8,10,12,25\rangle$. Then  $\AP(S,8)=\{0, 25, 10, 35, 12, 37, 22, 47\}$ and $G_1=\{x_1^3-x_3^2, x_2^5-x_4^2, x_1x_3-x_2^2\}$ is a minimal generating set (the reduced Gr\"obner basis) for $I_S$.
	We can easily see that $\AP(S,8)$ is a homogeneous set, while $x_2^5-x_4^2$ is a non-homogeneous element without $x_1$ in its support. Note that  $2\times 25=5\times10=8+3\times 10+12$. Hence replacing $x_2^5-x_4^2$ by $x_1x_2^3x_3-x_2^5$ and $x_1x_2^3x_3-x_4^2$, we get the minimal generating set $G_2=\{x_1^3-x_3^2, x_1x_2^3x_3-x_2^5, x_1x_2^3x_3-x_4^2, x_1x_3-x_2^2\}$  which is also a Gr\"obner basis and  satisfies the properties (3) and (5) of Theorem~\ref{**}.
\end{ex}

By a general result due to Robbiano (see \cite{Ro}, \cite{HRV}), Betti numbers of the associated graded ring $G(S)$ are upper bounds for Betti numbers of the semigroup ring $R$ i.e. $\beta_i(R)\leq\beta_i(G(S))$ for all $i\geq 1$.

\begin{dfn}
The semigroup $S$ is called \emph{of homogeneous type} if $\beta_i(R)=\beta_i(G(S))$ for all $i\geq 1$.
\end{dfn}
\begin{rem}\label{ci}
If  $S$  is of homogeneous type, then  $\depth(G(S))=\depth(R)=1$ and so $G(S)$ is Cohen-Macaulay.
\end{rem}
\begin{rem}\label{koszul}
	If   $G(S)$ is complete intersection, then  $\beta_1(R)=\beta_1(G(S))=d-1$.  Hence the free resolutions of $R$ and $G(S)$ coincide with  the Koszul complexes with respect to $d-1$ elements and so $S$ is of homogeneous type.
\end{rem}

The following  result is inspired  by the ideas given in the proof of \cite[Theorem 1.4]{HS}.
\begin{thm}\label{homogeneous type}
Let $S$ be a homogeneous numerical semigroup  with  Cohen-Macaulay tangent cone. Then $S$ is of homogeneous type.
\end{thm}
\begin{proof}
Let $E=\{f_1,\ldots,f_t,g_1,\ldots,g_r\}$ be the minimal set of generators of $I_S$ which exists by Theorem~\ref{**}(4) and it is also a standard basis. Let $f_1,\ldots,f_t$ be  homogeneous binomials and $g_1,\ldots,g_r$ be non-homogeneous. Then
the term of $g_i$ which is not the leading term is divisible by $x_1$, for $i=1,\ldots,r$. Hence $\pi(I_S)$ is a homogeneous ideal and so
\[\beta_i(\bar{P}/\pi(I_S))=\beta_i(\gr_{\bar{\fn}}(\bar{P}/\pi({I}_S)),\]
where $\bar{\fn}=\pi(\fn)$. Note that $G(S)$ is indeed the completion of $\gr_\fn(P/I_S)$ with respect to the $\fm$-adic topology. As $x_1$ is a non-zero-divisor on $G(S)$, it is also a regular element of
$\gr_\fn(P/I_S)$ and  $P/I_S$ as well. Hence
\[\gr_{\bar{\fn}}(\bar{P}/\pi({I}_S))\cong\gr_\fn(P/I_S)/x_1\gr_\fn(P/I_S),\]
by \cite[Lemma, P. 185]{H}.
 Now, the result follows from the fact that Betti numbers are preserved under dividing by regular elements.
\end{proof}

\begin{cor}\label{almost maximaleambedding}
Let $S$ be a numerical semigroup ring with almost embedding dimension. Then, $S$ is of homogeneous type if and only if it has reduction number two.
\end{cor}
\begin{proof}
By Example \ref{maximal embedding} we know that $S$ is homogeneous. Then, it is of homogeneous type if and only if $G(S)$ is Cohen-Macaulay. By \cite[Theorem 2.1]{RoVa} this happens if and only if the reduction number is two.
\end{proof}

\begin{rem}
	A Cohen-Macaulay local ring $(A,\fm)$ is \textit{stretched}, in the sense of \cite{Sa}, if and only if  $A$ admits an artinian reduction $(B, \fn)$ such that $\fn^2$ is principal.
Let $S$ be a numerical semigroup. Then $\kk[\![S]\!]$ is stretched if and only if $\AP(S,n_1)$ has a unique element of order two. If the order of all elements in the Ap\'ery set is at most two, then it has  almost maximal embedding dimension and so it is homogeneous. Assume that $\AP(S,n_1)$ has an element with order greater than two. Set $w\in \AP(S,n_1)$ the only element of order two. It's easy to see that $w=2n_{i_1}$ where $n_{i_1}$ is a minimal generator of $S$. Now, let $s=\sum^d_{i=2}r_in_i\in\AP(S,n_1)$ with $\ord_S(s)=\sum^d_{i=2}r_i>2$. As any subexpression of $s$ is again in the Ap\'ery set, any maximal subexpression of $s$ with length 2 should be equal to $w$. In particular $n_{i_1}=n_{i_2}$ and $s=ln_{i_1}$, where $l=\ord_S(s)$. Therefore $\{s\in\AP(S,n_1) ; \ord_S(s)\geq 2\}$ has only one maximal element with respect to $\preceq$. Now assume that its Cohen-Macaulay type is equal to $d-1$. The Cohen-Macaulay type of $S$ is exactly the number of maximal elements of $\AP(S,n_1)$ hence $\Max_{\preceq}\AP(S,n_1)=\{tn_{i_1}, n_2,\ldots,n_d\}\setminus\{n_{i_1}\}$, where $t$ is the maximal order of elements of the Ap\'ery set. In particular,  $S$ is homogeneous and so it is of homogeneous type if and only if the associated graded ring $G(S)$ is Cohen-Macaulay. This recovers \cite[Example~3.5]{RS} in the case $A$ is a numerical semigroup ring.


\end{rem}

It is proved in \cite[Proposition 2.5]{HS} that a numerical semigroup generated by an arithmetic sequence is of homogeneous type.
For some classes of  semigroups generated by  generalized arithmetic sequences, it is shown  in \cite[Corollary 4.12]{SZ}, that they are of homogeneous type. Now, we have:

\begin{cor}\label{hh}
Let $S$ be a numerical semigroup generated by a generalized arithmetic sequence. Then $S$ is of homogeneous type.
\end{cor}
\begin{proof}
By \cite[Corollary 3.2]{SZ}, $G(S)$ is Cohen-Macaulay. Now, the result follows by Remark~\ref{arithmetic} and Theorem~\ref{homogeneous type}.
\end{proof}

The following example shows that the converse of Theorem~\ref{homogeneous type}, does not hold even in embedding dimension three.

\begin{ex}\label{example}
Let $S:=\langle15, 21, 28\rangle$. Then $I(S)=(x_2^4-x_3^3, x_1^7-x_2^5)$  is minimally generated by a standard basis of two elements. Hence $G(S)$ is complete intersection and so $S$ is of homogeneous type (cf. Remark~\ref{koszul}), but
it is not homogeneous, since $3\times28=4\times21=84\in\AP(S, 15)$.
\end{ex}

\section{Small embedding dimensions}

We first recall the following definition: a binomial $x_i^{c_i}-\prod_{j\neq i}x_j^{u_{ij}}\in I_S$ is called \emph{critical} with respect to $x_i$ if $c_i$ is the smallest integer such that $c_in_i\in\langle n_1,\ldots,\widehat{n_i},\ldots,n_d\rangle$.
The notation $c_i$ will be used frequently in the rest of the section.
The \emph{critical ideal} of $S$, denoted by $C_S$, is the ideal of $\kk[x_1,\ldots,x_d]$ generated by all  critical binomials of $I_S$ (cf.~\cite{KO}).

\medskip

Recall that a numerical semigroup with Frobenius number $F(S)$,  is called \emph{irreducible}, if
it cannot be written as the intersection of two numerical semigroups
properly containing it. Let $S$ be an  irreducible numerical
semigroup. Then $S$ is called \emph{symmetric} if $F(S)$ is  odd and
it is called \emph{pseudo symmetric} if $F(S)$ is even (cf.~\cite{GR}).

\medskip

The following classical result by J. Herzog describes the minimal systems of generators of $I_S$ for numerical semigroups $S$ with embedding dimension three:

\begin{thm}[\text{\cite[Section 3]{H1}}]\label{Herzog}
	Let $S$ be a numerical semigroup of embedding dimension three. Then the following statements hold for some non-negative integers $c_{ij}$, $1\leq i,j\leq 3$.
	\begin{enumerate}
		\item If $S$ is symmetric, then after a permutation $(i,j,k)$ of $(1,2,3)$,  we have
		\[I_S=(x_i^{c_i}-x_j^{c_j}, x_k^{c_k}-x_i^{c_{ki}}x_j^{c_{kj}}).\]
		\item If $S$ is not symmetric, then \[I_S=(x_1^{c_1}-x_2^{c_{12}}x_3^{c_{13}}, x_2^{c_2}-x_1^{c_{21}}x_3^{c_{23}},x_3^{c_3}-x_1^{c_{31}}x_2^{c_{32}}),\]
	where $c_i=c_{ji}+c_{ki}$ for all permutation $(i,j,k)$ of $(1,2,3)$ and $c_{ij}>0$ for all $1\leq i\neq j\leq 3$.
	\end{enumerate}
\end{thm}

\begin{thm}\label{dim3.h}
Let $S$ be a numerical semigroup of embedding dimension three.
\begin{enumerate}
		\item Let $S$ be symmetric and $(i,j,k)$ be the permutation of $(1,2,3)$ such that the statement (1) of Theorem~\ref{Herzog} holds. Then
	$\AP(S,n_i)$ and $\AP(S,n_j)$ are homogeneous sets, and $\AP(S,n_k)$ is not homogeneous.
	\item If $S$ is not symmetric, then $\AP(S,n_i)$ is homogeneous for all $i=1,2,3$.
	\end{enumerate}
\end{thm}
\begin{proof}
	(1): From Theorem~~\ref{Herzog}, we have  $I_S=(x_i^{c_i}-x_j^{c_j}, x_k^{c_k}-x_i^{c_{ki}}x_j^{c_{kj}})$. As $x_i, x_j$ belong to the support of all generators of $I_S$, $\AP(S,n_i)$ and $\AP(S,n_j)$ are homogeneous by Corollary~\ref{homog}. Note that $c_in_i=c_jn_j$. Hence $x_i^{c_i}-x_j^{c_j}$ is not a homogeneous binomial.  If $\AP(S,n_k)$ is homogeneous, then $c_in_i=c_jn_j$ is not in $\AP(S,n_k)$. Therefore
	\[c_in_i=c_jn_j=r_in_i+r_jn_j+r_kn_k,\] for some non-negative integers $r_i,r_j$ and $r_k>0$. According to the definition of $c_i$ and $c_j$, we get $r_i=r_j=0$ and so  $r_kn_k\in \langle n_i,n_j\rangle$. Let $r_k=c_k+s_k$. Then
	\[c_in_i=c_jn_j=r_kn_k=c_{ki}n_i+c_{kj}n_j+s_kn_k,\] which implies that $c_{ki}=c_{kj}=0$, a contradiction.
	
	\vspace{2mm}
	(2): By Theorem~~\ref{Herzog}, $x_i,x_j,x_k$ belong to the support of all generators of $I_S$ and the result follows from Examples~\ref{generic}.
\end{proof}

\begin{rem}
By (2) in the above theorem, we have that if $S$ is not symmetric, $S$ is always homogeneous and so $S$ is of homogeneous type if and only if $G(S)$ is Cohen-Macaulay. For instance, this is the case for $S=\langle 3,5,7 \rangle$, see \cite[Table 3.4]{RV}.
\end{rem}

\begin{rem}
In the symmetric case, $S$ is not necessarily homogeneous neither of homogeneous type. For instance, this is what happens for $S=\langle 7,8,20 \rangle$, see \cite[Table 3.4]{RV}.
\end{rem}

We have seen two different classes of  numerical semigroups of homogeneous type:
homogeneous numerical semigroups which are not symmetric with Cohen-Macaulay tangent cones   (cf.~Theorem~\ref{homogeneous type}) and numerical semigroups with complete intersection tangent cones (cf.~Remark~\ref{ci}).
Our next result  shows that in embedding dimension three, these two (different) classes determine all numerical semigroups of homogeneous type.

\begin{thm}\label{dim3}
Let $S$ be a numerical semigroup with embedding dimension three. Then the following statements are equivalent.
\begin{enumerate}
\item $S$ is of  homogeneous type.
\item $\beta_1(R)=\beta_1(G(S))$.
\item $G(S)$ is Cohen-Macaulay, and either $S$ is homogeneous or $(I_S)_*$ is generated by pure powers of $x_2$ and $x_3$.
\item Either  $S$ is non-symmetric homogeneous with Cohen-Macaulay tangent cone, or $G(S)$ is complete intersection.
\end{enumerate}
\end{thm}
\begin{proof}
(1)$\Rightarrow$(2) is clear.

\vspace{2mm}
(2)$\Rightarrow$(3):
By Theorem~\ref{Herzog}, $\beta_1(R)\leq 3$. Therefore $\beta_1(G(S))\leq 3$ and so $G(S)$ is Cohen-Macaulay from \cite[Proposition 3.3]{RV}.

Assume that $S$ is not homogeneous. Hence,  Theorem~\ref{dim3.h} implies that $S$  is symmetric  and $k=1$ in the statements of Theorem~\ref{Herzog}. Therefore
\[I_S=(f_1:=x_2^{c_2}-x_3^{c_3}, f_2:=x_1^{c_1}-x_2^{c_{12}}x_3^{c_{13}}).\]
From \cite[Corollary 3.2]{Sh}, it follows that
$\{f_1,f_2,f_3:=x_2^{c_2+c_{12}}-x_1^{c_1}x_3^{c_3-c_{13}}\}$ is a standard basis for $I_S$,
$(I_S)_*=(x_3^{c_3}, x_2^{c_{12}}x_3^{c_{13}}, (f_3)_*)$ and we may assume that $c_{13}<c_3$. Since $c_{13}<c_3$, we can not remove neither  $x_3^{c_3}$ nor $x_2^{c_{12}}x_3^{c_{13}}$
from the  set of generators of $(I_S)_*$. On the other hand $\beta_1(R)=\beta_1(G(S))=2$. Therefore $(f_3)_*\in(x_3^{c_3}, x_2^{c_{12}}x_3^{c_{13}})$.
Note that
\[(f_3)_*=\left\{\begin{array}{ll}
f_3 & \mbox{ if } c_2+c_{12}=c_1+c_3-c_{13}\\
x_2^{c_2+c_{12}} & \mbox{ otherwise }\end{array}\right.\]
Since $(x_3^{c_3}, x_2^{c_{12}}x_3^{c_{13}})$ is a monomial ideal, if $(f_3)_*=f_3$, then each monomial term of $f_3$ belongs to this ideal.
Hence $x_2^{c_2+c_{12}}\in(x_3^{c_3}, x_2^{c_{12}}x_3^{c_{13}})$ and consequently $c_{13}=0$, i.e. $(I_S)_*=(x_3^{c_3}, x_2^{c_{12}})$.

\vspace{2mm}
(3)$\Rightarrow$(4): It is clear.

\vspace{2mm}
(4)$\Rightarrow$(1): It follows by Theorem~\ref{homogeneous type} and Remark~\ref{koszul}.
\end{proof}

Now, we look at numerical semigroups $S$ with embedding dimension four. We start by observing that, as in the case of embedding dimension $3$, $S$ is not necessarily homogeneous neither of homogeneous type. The following example is taken from \cite[Remark 3.10]{DMS}: let $S=\langle 16, 18, 21, 27 \rangle$. Then, $S$ is a complete intersection and $G(S)$ is Gorenstein but $G(S)$ is not a complete intersection, hence the minimal number of generators of the corresponding defining ideals are different and so $S$ is not of homogeneous type. Since $G(S)$ is Cohen-Macaulay, $S$ cannot be homogeneous neither. In fact, $81 \in \AP(S,16)$ and $81 = 3 \times 27 = 3 \times 18 + 27$ are two expressions with different size.


Let $S$ with embedding dimension four. Let $C_S$ be the critical ideal of $S$ as defined at the beginning of this section. By a result of A. Katsabekis and I. Ojeda, one can find a minimal system of generators of $I_S$ with the following special property:

\begin{prop}[\text{\cite[Proposition 3.9]{KO}}]\label{4g}
	Let $S$ be a numerical semigroup of embedding dimension four. Then there exists a minimal system of generators $E=E_1\cup E_2$ of $I_S$, where $E_1$ is minimal set of generators of $C_S$  and $E_2$ is a set of binomials with full support.
\end{prop}

\begin{thm}\label{ci}
Let $S$ be a numerical semigroup with embedding dimension four. Then
the following statements are equivalent for any $i=1,\ldots,4$.
\begin{enumerate}
	\item $\AP(S,n_i)$ is homogeneous.
	\item $\{c_jn_j \mid j\neq i\}\cap\AP(S,n_i)$ is a homogeneous set.
\end{enumerate}
\end{thm}
\begin{proof}

(1)$\Rightarrow$(2): It is clear.

	\vspace{2mm}
(2)$\Rightarrow$(1):
From Corollary~\ref{homog}, we only need to find a set of generators for $I_S$ such that all of its  non-homogeneous elements have $x_i$ in their support. Let $E$ be the minimal set of generators which exists by Proposition~\ref{4g}. We only need to check the property for the elements in $E_1$.
 Let $f_j:=x_j^{c_j}-x^{\b_j}$  be a non-homogeneous element for some $j\neq i$. Then $c_jn_j\notin\AP(S,n_i)$ and so $c_j=n_i+s$ for some $s\in S$. Now,  we can replace $f_j$ with two binomials $x_j^{c_j}-x_ix^{\a}$ and $x_ix^{\a}-x^{\b_j}$, where $\a$ is a factorization of $s$.
\end{proof}

\begin{cor}\label{c4}
	Let $S$ be a numerical semigroup with embedding dimension four. Then
	the following statements are equivalent.
	\begin{enumerate}
		\item  $S$ is homogeneous.
		\item $\{c_2n_2,c_3n_3,c_4n_4\}\cap\AP(S,n_1)$ is a homogeneous set.
		\item $c_2n_2$ and $c_4n_4$ are not in $\AP(S,n_1)$ and, if $c_3n_3\in\AP(S,n_1)$, then  $\{c_3n_3\}$ is homogeneous.
	\end{enumerate}
\end{cor}
\begin{proof}
	(1)$\Rightarrow$(2): It is clear since $\AP(S,n_1)$ is homogeneous.
	
	\vspace{2mm}
	(2)$\Rightarrow$(3): If $c_2n_2\in\AP(S,n_1)$, then  $c_2n_2\in\langle n_3,n_4\rangle$. Therefore $c_2n_2=r_3n_3+r_4n_4$ for some non-negative integers $r_3,r_4$. Since $n_2<n_3,n_4$, we have $c_2>r_3+r_4$, a contradiction. As $n_4>n_2,n_3$, a similar argument shows that $c_4n_4$ is not in $\AP(S,n_1)$.
	
    \vspace{2mm}
	(3) $\Rightarrow$(1): It follows from Theorem~\ref{ci}.
\end{proof}

The following well known result by H. Bresinsnky provides the systems of generators for the defining ideals of non-complete intersection symmetric numerical semigroups with embedding dimension four:

\begin{thm}[\text{Bresinsky's Theorem \cite[Theorem 3]{B}}]\label{B}
	Let $S$ be a numerical semigroup of embedding dimension four.
	Then  $S$ is symmetric and non-complete intersection if and only if,
	after permuting variables, if necessary,  $I_S$ is generated by the set
	\[\begin{array}{ll}G=&\{f_1=x_1^{c_1}-x^{c_{13}}_3x_4^{c_{14}}, f_2=x_2^{c_2}-x_1^{c_{21}}x_4^{c_{24}}, f_3=x_3^{c_3}-x_1^{c_{31}}x_2^{c_{32}},\\ & \ f_4=x_4^{c_4}-x_2^{c_{42}}x_3^{c_{43}}, f_5=x_3^{c_{43}}x_1^{c_{21}}-x_2^{c_{32}}x_4^{c_{14}}\},\end{array}\]
	where  $0<c_{ij}<c_j$.
\end{thm}

\begin{rem}\label{UG}
	Let $S$ be a symmetric numerical semigroup of embedding dimension four. If $S$ is not complete intersection,  then the set $G$ given in Theorem~\ref{B}, is the unique minimal system of binomial  generators for $I_S$ (cf. \cite[Corollary 3.15]{KO}).
\end{rem}

\begin{prop}
	Let $S$ be a symmetric and non-complete intersection numerical semigroup with embedding dimension four. Using the notation of Theorem~\ref{B}, the following statements hold.
	\begin{enumerate}
		\item $\AP(S,n_1)$ is homogeneous if and only if $f_4$ is a homogeneous polynomial.
		\item For each $i=2,3,4$, $\AP(S,n_i)$ is homogeneous if and only if $f_{i-1}$ is a homogeneous polynomial.
	\end{enumerate}
\end{prop}
\begin{proof}
	As $G$ is the unique minimal system of generators by Remark~\ref{UG}, the critical elements $c_in_i$ belong to $\AP(S,n_j)$ for $x_j\notin\supp(f_i)$.  Now, the result follows from Theorem~\ref{ci}.
\end{proof}

\begin{ex}
Let $S=\langle 8,13,15,17 \rangle$. Then, $I_S=(f_1=x_1^4-x_4x_3, f_2=x_2^3-x_1x_3, f_3=x_4^2-x_1x_2^2, f_4=x_3^2-x_2x_4, f_5=x_2^2x_3-x_1^3x_4)$ (with order $8,13,17,15$). Since $f_4$ is homogeneous, by the above proposition $S$ is homogeneous. We also have that $G(S)$ is not Cohen-Macaulay because $15+17= 4 \times 8$, hence $S$ is not of homogeneous type.
\end{ex}

The following two families extracted from \cite{GR} can be used to produce symmetric numerical semigroups $S$ with embedding dimension four and given multiplicity $m$, which are not of homogeneous type neither homogeneous:

\begin{lem}\label{sym}
	Let $q$ be a positive integer.
	\begin{enumerate}
		\item  If $m=2q+4$, then the numerical semigroup generated by \[\{n_1:=m,n_2:=m+1,n_3:=(q+1)m-2,n_4:=(q+1)m-1\},\] is a symmetric numerical semigroup of embedding dimension four and Frobenius number $F(S)=2qm+2q+1$.
		\item If  $m=2q+5$. Then the numerical semigroup generated by
		\[\{n_1:=m,n_2:=m+1,n_3:=(q+1)m+q+2,n_4:=(q+1)m+q+3\},\] is a  symmetric numerical semigroup of embedding dimension four, with Frobenius number $F(S)=2(q+1)m-1$.
	\end{enumerate}
\end{lem}
\begin{proof}
	(1) follows from \cite[Lemma 4.22]{GR} and (2) is the subject of \cite[Lemma 4.23]{GR}.
\end{proof}

\begin{prop}
	Let $S$ be a symmetric numerical semigroup of embedding dimension four, with one of the structures of Lemma~\ref{sym}. Then $S$ is not of homogeneous type and  $\AP(S,n_i)$ is not homogeneous for all $i=1,\ldots,4$.
\end{prop}
\begin{proof}
	In the case $m=2q+4$, we show that $I_S$ is generated by
	\[G_1:=\{x_1^{q+2}-x_2x_4, x_4^2-x_1^{q+1}x_3, x_2^{2q+1}-x_3x_4,x_3^2-x_1x_2^{2q}, x_1x_4-x_2x_3\}.\]
	First, we check that $G_1\subseteq I_S$:
	\begin{align*}
	 (q+2)n_1&=(q+2)m=(q+1)m+m+1-1=n_2+n_4,\\
	 (2q+1)n_2&=(2q+1)(m+1)=(2q+1)m+2q+1=2qm+m+2q+1\\ &=2qm+m+m-3=2(q+1)m-3=n_3+n_4, \\
	 2n_3&=2(q+1)m-4=2qm+2m-4=2qm+m+2q=m+2q(m+1) \\ &=n_1+2qn_2,\\
	 2n_4&=2(q+1)m-2=(q+1)m+(q+1)m-2=(q+1)n_1+n_3,\\
	 n_1+n_4&=m+(q+1)m-1=m+1-2+(q+1)m=n_2+n_3.
	\end{align*}
Hence $c_3=c_4=2$. From the proof of \cite[Lemma~4.22]{GR}, $n_3+n_4=(2q+1)n_2\in\AP(S,n_1)$. Therefore $c_2=2q+1$. Note that $n_2+n_4-n_3= m+2\notin S$. So that $(q+1)n_1=n_2+n_4$ has unique expression, in particular $c_1=q+1$.
As the last relation $x_1x_4-x_2x_3$ is not generated by the others, $I_S$ has more than 4 generators and so $S$ is not complete intersection. Now, Theorem~\ref{B} implies that $G_1$ is the minimal generating set of $I_S$. Note that none of the critical binomials in $G_1$ are homogeneous, so it follows that $\AP(S,n_i)$ is not homogeneous for all $i=1,\ldots,4$, by Theorem~\ref{ci}.
Now,  $f:=x_2^{2q+2}-x_3x_1^{q+2}=x_2(x_2^{2q+1}-x_3x_4)-x_3(x_1^{q+2}-x_2x_4)\in I_S$, where $2q+2\geq q+3$. Hence $f_* = x_3x_1^{q+2}$ is not generated by the elements of $(G_1)_*=\{x_2x_4, x_4^2,x_3x_4,x_3^2, x_1x_4-x_2x_3\}$. Therefore $G_1$ is not a standard basis and so $\beta_1(G(S))>\beta_1(R)$.

\medskip

In the second case $m=2q+5$, we show with a similar argument that
\[G_2:=\{x_1^{2q+3}-x_3x_4, x_2^{q+2}-x_1x_3, x_3^2-x_2^{q+1}x_4, x_4^2-x_2x_1^{2q+2},x_1x_4-x_2x_3\},\]
is the minimal set of generators of $I_S$.  Since all critical binomials in $G_2$ are non-homogeneous, $\AP(S,n_i)$ is not homogeneous for all $i=1,\ldots,4$, by Theorem~\ref{ci}.  Finally, the relation $x_2^{q+2}-x_1x_3$ implies that $n_3$ is a torsion element and $G(S)$ is not Cohen-Macaulay. In particular, $S$ is not of homogeneous type.
\end{proof}

For the pseudo symmetric case, we get  generators of the defining ideal  $I_S$ from the following result by J. Komeda:

\begin{thm}[\text{Komeda's Theorem \cite[Theorem 6.4]{K}}]\label{K}
     Let $S$ be a pseudo symmetric numerical semigroup of embedding dimension four.  After permuting variables, if necessary,
      the ideal $I_S$ is generated by
     \[\begin{array}{ll}G=&\{f_1=x_1^{c_1}-x_3x_4^{c_{4}-1}, f_2=x_2^{c_2}-x_1^{c_{21}}x_4, f_3=x_3^{c_3}-x_1^{c_1-c_{21}-1}x_2,\\ & \ f_4=x_4^{c_4}-x_1x_2^{c_{2}-1}x_3^{c_{3}-1}, f_5=x_3^{c_{3}-1}x_1^{c_{21}+1}-x_2x_4^{c_{4}-1}\},\end{array}\]
     where $0<c_{21}<c_1$.
\end{thm}

\begin{prop}\label{KK}
	Let $S$ be a pseudo symmetric numerical semigroup of embedding dimension four. Using the notation of Komeda's Theorem, we have
	\begin{enumerate}
		\item $\AP(S,n_1)$ is a homogeneous set.
		\item For $i=2,3,4$, $\AP(S,n_i)$ is homogeneous if and only if  $\{c_{i-1}n_{i-1}\}\cap\AP(S,n_i)$ is a homogeneous set.
		\item For each $i=2,3,4$, if $f_{i-1}$ is a homogeneous polynomial, then $\AP(S,n_i)$ is a  homogeneous set. The converse holds if $G$ is the only set of minimal generators for $I_S$.
	\end{enumerate}
\end{prop}
\begin{proof}
 (1):  Since $x_1\in\supp(f_i)$ for $i=2,3,4$, it follows by Corollary~\ref{homog}.

 \vspace{2mm}
(2): Since $x_i\in\supp(f_j)$ for $i=2,3,4$, $j\neq i-1$, we have $c_in_i\notin\AP(S,n_j)$  for  $i=2,3,4$, $j\neq i-1$. Now the result follows by Corollary~\ref{c4}.

\vspace{2mm}
(3): If $f_{i-1}$ is homogeneous, then  $\AP(S,n_i)$ is homogeneous by the statement~(2). If $f_j$ is  uniquely determined for $j=1,\dots,4$, then  $c_{i-1}n_{i-1}\in\AP(S,n_i)$ for $i=2,3,4$  and the result follows from the statement~(2).
\end{proof}

Similarly to the symmetric case, we have the following families of pseudo symmetric numerical semigroups with embedding dimension four and given multiplicity $m$:

\begin{lem}\label{ps}
	Let $q$ be a positive integer.
	\begin{enumerate}
		\item  If $m=2q+5$, then the numerical semigroup generated by \[\{n_1:=m,n_2:=m+1,n_3:=(q+1)m+q+2,n_4:=(q+1)m+m-1\},\] is a pseudo symmetric numerical semigroup of embedding dimension four and Frobenius number $F(S)=2(q+1)m-2$.
		\item If  $m=2q+4$. Then the numerical semigroup generated by
		\[\{n_1:=m,n_2:=m+1,n_3:=qm+2q+3,n_4:=(q+1)m+q+2\},\] is a pseudo symmetric numerical semigroup of embedding dimension four, with Frobenius number $F(S)=2qm+2q+2$.
	\end{enumerate}
\end{lem}
\begin{proof}
(1) follows from \cite[Lemma 4.29]{GR} and (2) is the subject of \cite[Lemma 4.30]{GR}.
\end{proof}

\begin{ex}
	Let $S=\langle 7,8,17,20\rangle$. It is pseudo symmetric from Lemma~\ref{ps}(1), considering $q=1$. The defining ideal $I_S$ is generated by $\{x_1^4-x_2x_4,x_2^3-x_1x_3,  x_3^2-x_1^2x_4, x_4^2-x_1x_2^2x_3, x_1^3x_2^2-x_3x_4\}$, so that $S$ is homogeneous by Proposition~\ref{KK}. Note that $G(S)$ is not Cohen-Macaulay because of $7+17=3\times 8$ hence $S$ is not of homogeneous type too.
\end{ex}

\begin{ex}
	Let $S=\langle 6,7,11,15\rangle$. Then $S$ is  pseudo symmetric from Lemma~\ref{ps}(2), considering $q=1$. The defining ideal $I_S$ is generated  by $\{ x_1^3-x_2x_3, x_2^3 - x_1x_4, x_3^2 - x_2x_4, x_1^2x_2^2 - x_3x_4, x_1^2x_2x_3-x_4^2\}$.
	So this ordering $6,7,11,15$ is not the permutation that gives the generators in Theorem~\ref{KK}, but  $S$ is homogeneous by Corollary~\ref{homog}.
	As  $6 + 15 = 3 \times 7$, $G(S)$ is not Cohen-Macaulay and so $S$ is not of homogeneous type.
\end{ex}

\begin{ex}
	Let $S=\langle 8,9,23,28 \rangle$. $S$ is pseudo symmetric from Lemma~\ref{ps}(2), considering $q=2$. The defining ideal $I_S$ is generated by $\{f_1=x_1^4-x_2x_3, f_2=x_2^4-x_1x_4,  f_3=x_3^2-x_2^2x_4, f_4=x_4^2-x_1^3x_2x_3, f_5=x_1^3x_2^3-x_3x_4\}$. The non-homogeneous binomial $f_3$ comes from the relation $2n_3=2n_2+n_4=46\in\AP(S,8) $. Hence $S$ is not homogeneous. Note that $G(S)$ is not Cohen-Macaulay because of relation $f_2$ and so $S$ is not of homogeneous type.  Now, considering the permutation $9,23,8,28$, we get 	
	$G=\{x_1^4-x_3x_4,  x_2^2-x_1^2x_4,x_3^4-x_1x_2,x_4^2-x_1x_2x_3^3,x_1^3x_3^3-x_2x_4\}$, as in Theorem~\ref{K} and so $\AP(S,9)$ is homogeneous by Proposition~\ref{KK}.
\end{ex}

These examples can be extended in the following way:

\begin{prop}
	Let $S$ be a pseudo symmetric numerical semigroup of embedding dimension four, with the structure of Lemma~\ref{ps}. If  $m=2q+5$, then $S$ is homogeneous. If $m=2q+4$ and $q>1$, then $S$ is not homogeneous and $G(S)$ is not Cohen-Macaulay.
\end{prop}
\begin{proof}
	Let $m=2q+5$, then $2n_3=2(q+1)m+2q+4>2(q+1)m+2q+3=2(q+1)m-2+m=F(S)+m$. Therefore
	$2n_4>n_3+n_4>2n_3>F(S)+m$ and so $2n_4, n_3+n_4, 2n_3$ are not in $\AP(S,m)$.  This implies that $\{c_3n_3,c_4n_4\}\cap\AP(S,m)$ is empty or has elements of order one and so it is a homogeneous set. If $c_2n_2\in\AP(S,m)$ and $c_2n_2=r_3n_3+r_4n_4$, as $n_3+n_4\notin\AP(S,m)$ we should have $r_3=0$ or $r_4=0$. Then $r_3,r_4\leq 1$ which is not possible. Hence $\AP(S,m)$ is homogeneous by Corollary~\ref{c4}.
	
\medskip
	
	Now, let $m=2q+4$. Then $2n_3-n_1=2(qm+2q+3)-m=2(qm+2q+3)-2q-4=2qm+2q+2=F(S)\notin S$. Hence $2n_3\in\AP(S,n_1)$.  On the other hand, as in the proof of \cite[Lemma~4.30]{GR}, the only elements of order two in  $\AP(S,m)$ belong to $\{2n_2, n_2+n_4\}$. If $2n_3$ is homogeneous, then $2n_3=n_2+n_4$. Hence
	$2qm+4q+6=m+1+(q+1)m+q+2$, which implies that $2m=qm+3q+3$, a contradiction since $q>1$.
	
	For the last statement, note that $n_4+n_1=(q+1)m+q+2+m=(q+2)(m+1)=(q+2)n_2$ and so $n_4$ is a torsion element of $S$.
	\end{proof}

In embedding dimension three, we have classified all numerical semigroups of homogeneous type in to numerical semigroups with complete intersection tangent cones and the homogeneous ones with Cohen-Macaulay tangent cones (cf. Theorem~\ref{dim3}). On other hand, all four generated numerical semigroups of homogeneous type that we have discussed in this section are homogeneous. So it would be natural to look for a similar classification in larger embedding dimensions. Nevertheless, this is not true as the following example with embedding dimension $4$ (provided to us by F. Strazzanti) shows:

\begin{ex}
	Let $S=\langle 7,8,11,12 \rangle$. Then $I_S=(x_2x_4^2-x_1^3x_3, x_2x_3-x_1x_4, x_2^3-x_4^2, x_1x_2^2-x_3x_4, x_1^2x_2-x_3^2, x_1^4-x_2^2x_4)$.  Computing the  free resolutions of both $S$ and $G(S)$ by \cite{DGPS}, we find their  total Betti numbers
	\[\beta_1(R)=\beta_1(G(S))=6, \beta_2(R)=\beta(G(S))=8, \beta_3(R)=\beta(G(S))=3.\]
Therefore $S$ is of homogeneous type and $G(S)$ is not a complete intersection.  We also have $\AP(S,7)=\{0,8,11,12,16,20,24\}$ and that $24 = 8 \times 3 = 12 + 12$, so $S$ is not homogeneous
\end{ex}

\section{Gluing of homogeneous semigroups}

Throughout this section  $S_1$ and $S_2$  are two numerical semigroups generated by $m_1<\cdots<m_d$
 and $n_1<\cdots<n_k$, respectively.
 Let $p\in S_1$ and $q\in S_2$ be two positive integers  satisfying $\gcd(p,q)=1$ with $p\notin\{m_1,\ldots,m_d\}$ and
 $q\notin\{n_1,\ldots,n_k\}$. The numerical semigroup   $$S=<qm_1,\ldots,qm_d,pn_1,\ldots,pn_k>$$ is called a \emph{gluing} of $S_1$ and $S_2$ (cf. \cite{R}).
An \emph{extension} of $S_1$, is a gluing $S$ with $S_2=\NN$~(cf. \cite{AM}).

In the rest of this section, $S$ will denote the above gluing of $S_1$ and $S_2$.

\begin{dfn}\label{G}  Let $S$ be a gluing of $S_1$ and $S_2$.
	\begin{enumerate}
		\item $S$ is called  nice gluing,  if $q=an_1$ for some $1<a\leq \ord_{S_1}(p)$ (cf. \cite{AMS}).
		\item $S$ is called  specific gluing,   if $\ord_{S_2}(q)+l_q(S_2)\leq \ord_{S_1}(p)$, where
		$l_q(S_2)=\max\{\ord_{S_2}(s+q)-\ord_{S_2}(q)-\ord_{S_2}(s) ; s\in S_2\}$  (cf. \cite{JZ}).
	\end{enumerate}
\end{dfn}

\begin{rem}\label{m}
	It is easy to see from the definition that the set $$\{qm_1,\ldots,qm_d,pn_1,\ldots,pn_k\}$$ is a minimal system of
	generators of $S$ and $m(S)=\min\{qm_1,pn_1\}$. If $S$ is a nice  gluing of $S_1$ and $S_2$, then $qm_1=an_1m_1\leq
	\ord_{S_1}(p)n_1m_1\leq pn_1$ and so $m(S)=qm_1$. This is also the case when $S$ is a specific gluing, from \cite[Corollary~3.14]{JZ}.
\end{rem}

\begin{rem}\label{glu}
We consider $I_{S_1}$ as an ideal of $\kk[x_1,\ldots,x_d]$ and $I_{S_2}$ an ideal of $\kk[y_1,\ldots,y_k]$.
Then $I_S=I_{S_1}+I_{S_2}+\langle x^\p-y^\q\rangle$ is an ideal of $\kk[x_1,\dots,x_d,y_1,\dots,y_k]$, where $\p$ and $\q$ are, respectively, some  factorizations of $p$ and $q$ (cf.~\cite[Theorem~1.4]{R}).
\end{rem}

\begin{prop}\label{Apery}
	The following statements hold for all $s\in S_1$.
	\begin{enumerate}
		\item $\AP(S,qs)=\{qz_1+pz_2 \ ; \ z_1\in\AP(S_1,s) , z_2\in\AP(S_2,q)\}$.
		\item If $qz_1+pz_2\in\AP(S,qs)$, then $z_1\in\AP(S_1,s)$.
		\item If $p\notin\AP(S_1,s)$, then for all $z\in\AP(S,qs)$ there exist unique $z_1\in S_1$ and $z_2\in S_2$ such that $z=qz_1+pz_2$.
		\item If $qz_1+pz_2\in\AP(S,qs)$ and $p\notin\AP(S_1,s)$, then  $z_2\in\AP(S_2,q)$.
	\end{enumerate}
\end{prop}

\begin{proof}
	The statements 	(1) and (2) are the subject of \cite[Proposition 3.8]{JZ}.
	For the proof of (3), note that $z=qz_1+pz_2$ for some $z_1\in\AP(S_1,s)$ and $z_2\in\AP(S_2,q)$, by (1). If there exist another $z'_1\in S_1$ and $z'_2\in S_2$ with
	$z=qz'_1+pz'_2$, then  $z_1=z'_1+\alpha p$ for some integer $\alpha$, as 	$\gcd(p,q)=1$.  On the other hand, $z'_1$ and $z_1$ belong to $\AP(S_1,s)$ by (2). Hence $\alpha=0$ and so $z_1=z'_1$ and  $z'_2=z_2$.
	
	The statement (4) follows by (1) and (3).
\end{proof}

\begin{cor}\label{HS}
	If $\AP(S,qs)$ is homogeneous for some $s\in S_1$, then $\AP(S_1,s)$ and $\AP(S_2,q)$ are  also homogeneous.
\end{cor}

The following example shows that the gluing of two homogeneous numerical semigroup is not necessarily homogeneous.

\begin{ex}
	Let  $S:=\<15,21,28\>$. Then $S$ is an  extension  of $S_1=\<5,7\>$ with $q=3$ and $p=28$, but $S$ is not homogeneous from Example~\ref{example}.
\end{ex}

The following result is the key for our study of the homogeneity of a gluing:

\begin{thm}\label{thm-glu}
	Let $S$ be a gluing of $S_1$ and $S_2$,  $s\in S_1$ and $n=\min\{n\in\NN ; np\notin\AP(S_1,s)\}$. Then the following statements are equivalent.
	\begin{enumerate}
		\item $\AP(S, qs)$ is homogeneous.
		\item  $\AP(S_1,s)$ and $\AP(S_2,nq)$ are homogeneous, and if  $n>1$, then $\ord_{S_1}(p)=\ord_{S_2}(q)$.
	\end{enumerate}
\end{thm}
\begin{proof}
	(1)$\Rightarrow$(2): If $\AP(S,qs)$ is homogeneous, then $\AP(S_1,s)$ is homogeneous by Corollary~\ref{HS}. Let $y\in\AP(S_2,nq)$ with two expressions with different lengths.
	Then $py\notin\AP(S,qs)$ and so we can  write $py=qs+qs_1+pz_2$ for some $s_1\in S_1$ and $z_2\in S_2$. Let $z_1:=s+s_1$, then $z_1\notin\AP(S_1,s)$ and $py=qz_1+pz_2$.
	 As $\gcd(p,q)=1$, we have $z_1=\alpha p$ and $y=z_2+\alpha q$ for some integer $\alpha\geq 0$. Note that $\alpha\geq n$, since $\alpha p=z_1\notin\AP(S_1,s)$. Hence $y=z_2+\alpha q\notin\AP(S_2,nq)$, a contradiction.
	
	\vspace{2mm}
	Now, assume that  $\ord_{S_1}(p)\neq\ord_{S_2}(q)$. Then $\{qp\}$ is not a homogeneous set of $S$ and so $qp\notin\AP(S,qs)$.  Note that  $qp=qs+qs_1+pz_2$ for some $s_1\in S_1$ and $z_2\in S_2$. Setting $z_1:=s+s_1$ we get $qp=qz_1+pz_2$ and $z_1\notin\AP(S_1,s)$.
The condition $\gcd(p,q)=1$ implies that $z_1=ap$ and $z_2=bq$ for some non-negative integers $a$ and $b$. Hence $qp=(a+b)qp$ and so $a+b=1$. Since $z_1\notin\AP(S_1,s)$, we get $a=1$ and $b=0$. Therefore $p=z_1$ is not in $\AP(S_1,s)$ i.e. $n=1$.
	
	\vspace{2mm}
	(2)$\Rightarrow$(1): Let $\p$, $\q$ and $\a$ be factorizations of $p$, $q$ and $s$, respectively and let $E_1$ and $E_2$ be minimal generating sets for $I_{S_1}$ and $I_{S_2}$, respectively, as the ones in Proposition~\ref{homog-prop}(2).
	Now, $E=E_1\cup E_2\cup\{x^\p-y^\q\}$ is a generating set for $I_S$ by Remark~\ref{glu}. Note that  one term of each  non-homogeneous binomial $f\in E_1$ is divisible by $x^\a$, and any non-homogeneous binomial $g=y^\cc-y^\b\in E_2$ has one term divided by $y^{n\q}$. Assume that $y^\cc=y^{n\q}y^\d$. Then  $h:=x^{n\p}y^\d-y^\b\in I_S$ and $g$ is generated by $h$ and $x^\p-y^\q$. Therefore replacing $g$ by $h$, we get again a generating set for $I_S$. Continuing in this way, we get a generating set for $I_S$ which satisfies the property of
	Proposition~\ref{homog-prop}(3), since $np\notin\AP(S_1,s)$. Note that, if $x^\p-y^\q$ is non-homogeneous, then $n=1$ which means that $p\notin\AP(S_1,s)$.
	\end{proof}

As a consequence we obtain the necessary and sufficient conditions for an extension of a homogeneous numerical semigroup to be homogeneous:

\begin{cor}\label{pnotinAp}
	Let $S_1$ be a homogeneous numerical semigroup and $S_2=\NN$. If $qm_1<p$, then $S$ is homogeneous if and only if  one  of the following conditions hold.
	\begin{enumerate}
		\item $q=\ord_{S_1}(p)$.
		\item $p\notin\AP(S_1,m_1)$.
	\end{enumerate}
\end{cor}

\begin{cor}
	Let $S_1$ be a homogeneous numerical semigroup with Cohen-Macaulay tangent cone and $S_2=\NN$. For each positive integer $q$, if  $p\in S_1\setminus\AP(S_1,m_1)$ with $\ord_{S_1}(p)\geq q$, then $S$ is of homogeneous type.
\end{cor}
\begin{proof}
Note that	$S$ is a specific extension of $S_1$ and so $G(S)$ is Cohen-Macaulay by \cite[Theorem~3.16]{JZ}. On the other hand $S$ is homogeneous by Corollary~\ref{pnotinAp}. Now, Theorem~\ref{homogeneous type} implies the result.
\end{proof}

Now we want to show how to construct systematically non homogeneous numerical semigroups whose tangent cones are  complete intersection. The following auxiliary result is needed,  where we use the notations and concepts of \cite{AMS}.

\begin{lem}\label{glu.Grobner}
	Let $S=\langle qm_1,\ldots,qm_d,pn_1,\ldots,pn_k\rangle$  be a nice gluing of  $S_1$ and $S_2$. Let $G_1$ and $G_2$ be minimal standard  bases of $I_{S_1}$ and $I_{S_2}$, respectively. If $G(S_1)$ and $G(S_2)$ are Cohen-Macaulay, then $G = G_1\cup G_2\cup\{x^\p-y_1^{a}\}$ is a minimal standard bases of $I_S$, for some factorization $\p$ of $p$. In particular if $G(S_1)$ is complete intersection, then $G(S)$ is also complete intersection.
\end{lem}
\begin{proof}
It follows from the proof of \cite[Theorem 2.6]{AMS}.
\end{proof}

\begin{lem}\label{3}
	Let $S_1$ be a numerical semigroup of embedding dimension two and multiplicity $m(S_1)>3$. Then there exists an extension of $S_1$ with complete intersection tangent cone which is not homogeneous.
\end{lem}
\begin{proof}
	Let $S_1=\langle m_1,m_2\rangle$. Then $\AP(S_1,m_1)=\{m_2,\ldots,(m_1-1)m_2\}$. Let $p=(m_1-1)m_2$ and $q=2$. Then $\ord_{S_1}p=m_1-1>2=q$. Hence $S=\langle qm_1,qm_2,p\rangle$ is not homogeneous by Corollary~\ref{pnotinAp}. On the other hand $G(S)$ is complete intersection from Lemma~\ref{glu.Grobner}.
\end{proof}

\begin{prop}\label{d}
Let $d\geq 3$ be an integer. Then there exist infinitely many numerical semigroups  of embedding dimension $d$ with complete intersection tangent cones, which are not homogeneous.
\end{prop}
\begin{proof}
 For $d=3$, the result is clear by Lemma~\ref{3}. We proceed by induction on $d$. Let $S_1=\langle m_1<\cdots<m_{d-1}\rangle$ be a numerical semigroup of embedding dimension $d-1$, with complete intersection tangent cone which is not homogeneous. Let $q=2$ and $p\in S_1$  such that $\ord_{S_1}(p)\geq2$ and $\gcd(q,p)=1$. Then the extension $S:=\langle qm_1,\ldots, qm_{d-1},p\rangle$ is a nice  extension of $S$.
Since $S_1$ is not homogeneous,  $S$ is not homogeneous by Theorem~\ref{thm-glu}. More over  $G(S')$ is complete intersection by Lemma~\ref{glu.Grobner}.
\end{proof}

As a consequence we get:

\begin{cor}\label{d1}
Let $d\geq 3$. Then there exist infinitely many numerical semigroups with embedding dimension $d$, which are of homogeneous type but they are not homogeneous.
\end{cor}

\section{Shifted family of semigroups}

Let $\n: n_1< \cdots<n_d$ be a sequence of positive integers. For any non-negative integer $j$, we consider the shifted family
\[\n+j: n_1+j,\ldots,n_d+j.\] and the semigroup
\[S+j:=\langle  n_1+j,\ldots,n_d+j\rangle,\]
that we call the $j$-th shifting of $S$.

\begin{rem}
	If the semigroup $S$ generated by $\n$ is  a numerical semigroup, it may happen that $S+j$ is not anymore a numerical semigroup.
	For instance, let $S=\langle 4, 7\rangle$. Then $S+2=\langle 6, 9\rangle$. 	Also, it may happen that $\n$ is a minimal system of
	generators of $S$ but the shifted family is not anymore a minimal
	system of generators of $S+j$. For instance, $S=\langle 5, 11, 13\rangle$ and $S+1=\langle 6, 12, 14\rangle=\langle 6, 14\rangle$.
\end{rem}

\begin{lem}\label{min-shift}
	If $S$ is the numerical semigroup minimally generated by $\n$, then $S+j$ is minimally generated by $\n+j$ for all $j>n_d-2n_1$.
\end{lem}
\begin{proof}
    Assume  that $n_r+j=\sum^d_{i=1}s_i(n_i+j)$ for some non-negative integers $s_i$. Let $a=\sum^d_{i=1}s_i$. If $a\geq 2$, then
    \[n_r+j\geq a(n_1+j)\geq 2n_1+2j>n_d+j.\]
    Hence $n_r>n_d$, a contradiction.
\end{proof}

We will use the following  notation in the sequel:
\begin{itemize}
    \item $m_i:=n_d-n_i$ for $1\leq i\leq d$;
    \item $g:=\gcd(m_1,\ldots,m_{d-1})$;
    \item $T:=\<\frac{m_1}{g},\ldots,\frac{m_{d-1}}{g}\>$;
     \item $L:=m_1m_2\left(\frac{gc+m_1}{m_{d-1}}+d\right)-n_d$,
     where $c=\min\{s\in S; s+i\in S \mbox{ for all } i>0\}$ is the conductor of $T$.
 \end{itemize}

	Let $B:=\sum^d_{i=1}m_i+d+g$ and $N:=\max\{m_1(d+\reg(J(\n))), m_1m_2\left(\frac{gc+m_1}{m_{d-1}}+B\right)\}$, where $J(\n)$ is the ideal generated by all homogeneous elements of $I_S$ and $\reg(J(\n))$ is the Castelnuovo-Mumford regularity of $J(\n)$.
	It is shown by Vu, in \cite[Corollary 3.6]{V}, that for any $j>N$, for inhomogeneous  prime binomials $x^\a-x^\b$ of  $I_S$, $x_1$ divides $x^\a$ where $|\a|>|\b|$. Using this fact, Herzog and Stamate \cite[Theorem 1.4]{HS} show that
	$S+j$ is of homogeneous type, in particular  $G(S+j)$ is Cohen-Macaulay, for all $j>N$.
In the following result, we improve this bound $N$ by $L$, which only depends on the initial data of the family $\n$, using ideas inspired by \cite{V}.

\begin{thm}\label{shift}
Let $j>L$ and $s\in S+j$. If $\a,\a'\in\F(s)$ with $|\a|>|\a'|$,  then there exists $\b\in\F(s)$ such that $|\b|=|\a|$ and $b_1\neq 0$.
\end{thm}
\begin{proof}
 Let $l:=|\a|$ and $l':=|\a'|$. Then
\[l(n_d+j)-s=l(n_d+j)-\sum^d_{i=1}a_i(n_i+j)=l(n_d+j)-\sum^d_{i=1}a'_i(n_i+j).\]
Hence we have  $\sum^d_{i=1}a_i(n_d-n_i)=\sum^d_{i=1}a'_i(n_d-n_i)+(l-l')(n_d+j)$, equivalently
\begin{equation}\label{l-l'}
\sum^d_{i=1}a_im_i=\sum^d_{i=1}a'_im_i+(l-l')(n_d+j).
\end{equation}
If $d=2$, then the above equality (\ref{l-l'}) implies that
\[a_1(n_2-n_1)=a'_1(n_2-n_1)+(l-l')(n_2+j).\]
In particular $a_1\neq 0$ and the result follows. Now assume that $d>2$.
 Let $x:=\sum^d_{i=1}a_im_i$ and $\cc$ be a factorization of $x$ with $|\cc|=l$, in the semigroup $\<m_1,\ldots,m_{d-1}\>$,
 whose support has the largest cardinality among  all factorizations of  $x$ with total order $l$.
In our argument we  use several times the fact that $m_1>m_2>\cdots>m_{d-1}>m_d=0$. We show, arguing by contradiction, that $c_1\neq 0$. If $c_1=0$, then $x\leq lm_2$  and so
\begin{equation}
x/m_2\leq l.
\end{equation}
Let $F:=\{i; c_i\neq 0\}$ and $y:=x-\sum_{i\in F}m_i$. Note that
\[\frac{m_1m_2}{m_{d-1}}(gc+m_1)+dm_1m_2>gc+dm_1.\]
In particular
\[L+n_d=m_1m_2\left(\frac{gc+m_1}{m_{d-1}}+d\right)>gc+dm_1.\]
Now,
\[y\geq x-(d-1)m_{1}\geq n_d+j-(d-1)m_{1}>gc+m_1,\] the second inequality holds by (\ref{l-l'}), since $m_d=0$, and
the last one follows because
$n_d+j>L+n_d>gc+dm_{1}$.
As $y>gc+m_1$, we may write $y=tm_1+v$ for some integers $t>0$ and $gc\leq v<gc+m_1$. Note that both $y$ and $m_1$ are divisible by $g$ and so $g|v$.
Since $v/g\geq c$, $v/g\in T$ and so $v/g=\sum^{d-1}_{i=1}w_i(m_i/g)$ for some $w_i\geq 0$. Hence
\[y=tm_1+\sum^{d-1}_{i=1}w_im_i.\]
As $v<gc+m_1$ and $v=\sum^{d-1}_{i=1}w_im_i\geq\sum^{d-1}_{i=1}w_im_{d-1}$, we have
\begin{equation}\label{I}
\sum^{d-1}_{i=1}w_i<\frac{gc+m_1}{m_{d-1}}.
\end{equation}
Note that $n_d+j>n_d+L=m_1m_2\left(\frac{gc+m_1}{m_{d-1}}+d\right)$, therefore
\begin{equation}\label{III}
\frac{n_d+j}{m_1m_2}>\frac{gc+m_1}{m_{d-1}}+d.
\end{equation}
Now, we have
\[\begin{array}{ll}
\sum^{d-1}_{i=1}w_i+d+t&<\frac{gc+m_1}{m_{d-1}}+d+t,  \mbox{ from } (\ref{I})\\
&<\frac{gc+m_1}{m_{d-1}}+d+y/m_1\\
&<\frac{n_d+j}{m_1m_2}+y/m_1, \mbox{ from } (\ref{III})\\
&<\frac{x}{m_1m_2}+\frac{x}{m_1}=\frac{x(m_2+1)}{m_1m_2}\leq x/m_2\leq l.
\end{array}\]
On the other hand, $\sum^{d-1}_{i=1}m_i+gc+m_1<n_d+j\leq x$. Therefore we may write $x=\sum^d_{i=1}z_im_i$, where
$z_1=t+w_1$, $z_i=\delta_{i,F}+w_i$ for $i=2,\ldots,d-1$ and $z_d=l-(\sum^{d-1}_{i=1}w_i+t+|F|)$. Now, $\z$ is a factorization of $x$  with
larger support than $\cc$, a contradiction. Therefore $c_1\neq 0$. Note that $s=l(n_d+j)-x=\sum^d_{i=1}c_in_i$. Hence we can take $\b=\cc$ and the result follows.
\end{proof}

\begin{cor}\label{sh-h}
If  $j>L$, then  $S+j$ is homogeneous and $G(S)$ is Cohen-Macaulay.
\end{cor}
\begin{proof}
It follows by Theoremn~\ref{**}(3).
\end{proof}

\begin{cor}\cite{HS}
If  $j>L$, then  $S+j$ is  of homogeneous type.
\end{cor}
\begin{proof}
It follows by Corollary~\ref{sh-h} and Theorem~\ref{homogeneous type}.
\end{proof}

\begin{dfn}\label{sh-t}
Given a sequence of positive integers $\s: s_0 = 0 < s_1 < \cdots <  s_{d-1}$, $d\geq 2$, we say that $\n$ is of \emph{shifting type} $\s$ if $s_i= n_d-n_{d-i}$, for all $1 \leq \cdots \leq d-1$. We also say that the semigroup $S=\langle n_1, \dots , n_d \rangle$ is of shifting type $\s$.
\end{dfn}

\begin{rem}
Note that $n_{d-i} = s_{d-1}-s_i+n_1$ for all $i= 1, \dots , d$, hence the sequence $\n$ is uniquely determined by $n_1$ and its shifting type.
\end{rem}

\begin{rem}
The shifting type is invariant under shifting and two sequences of $d$ positive integers are shifted one from the other if and only if they have the same shifting type.
\end{rem}

Note that the sequence of integers $$\n: n_1=1, n_2=s_{d-1}-s_{d-2}+1, \dots , n_d=s_{d-1}-s_0+1=s_{d-1}+1$$ is the one with shifting type $\s$ and the smallest possible $n_1$.

\medskip

Now, for a given sequence $\s$, let:
\begin{itemize}
    \item $g:=\gcd(s_1,\ldots,s_{d-1})$;
    \item $T:=\<\frac{s_1}{g},\ldots,\frac{s_{d-1}}{g}\>$;
    \item $L:=s_{d-1}s_{d-2}\left(\frac{gc+ds_{d-1}}{s_1}+d\right)-s_{d-1}-1$,
     where $c$ is the conductor of $T$.
 \end{itemize}

\medskip

We may reformulate the above results in the following way:

\begin{prop}
Given a sequence of positive integers  $\s: s_0 = 0 < s_1 < \cdots <  s_{d-1}$, for any $e>L$ all the numerical semigroups $S=\langle n_1, \dots, n_d \rangle$ with $n_1=e$ and shifting type $\s$ are homogeneous and $G(S)$ is Cohen-Macaulay.
\end{prop}

By using the notation in \cite{HS}, we define the \emph{width} of a numerical semigroup $S$ as the difference between the largest and the smallest generator in a minimal set of generators of $S$, and denote this number as $\textrm{wd}(S)$. It is clear that for a given positive integer $w\geq 2$, there is only a finite number of possible sequences $\s$ for the shifting type of the numerical semigroups whose width is bounded by $w$. So we finally conclude that:

\begin{prop}
Let $w\geq 2$. Then, there exists a positive integer $W$ such that all the numerical semigroups $S$ with $\textrm{wd}(S)\leq w$ and multiplicity $e\geq W$, are homogeneous and $G(S)$ is Cohen-Macaulay.
\end{prop}

\begin{ex}
	Let $a<b$ be positive integers. Then $S_k=\<k, k+a, k+b\>$ is a numerical semigroup with shifting type $s_1=b-a, s_2=b$, for any $k>0$ and $L=b(b-a)(\frac{gc+b}{b-a}+3)-b-1=b(4b-3a)-1+b(gc-1)$.
\end{ex}

\begin{rem}
For numerical semigroups $S_k=\<k,k+a,k+b\>$ of embedding dimension three, the given bound by Vu \cite{V}, is improved by Stamate in \cite[Theorem 3.5]{St} showing that the Betti numbers of $S_k$ are periodic in $k$, for  $k>k_{a,b}=\max\{b(\frac{b-a}{g}-1),b\frac{a}{g}\}$. Moreover, $S_k$ is of homogeneous type for $k>k_{a,b}$. As the above example shows, in embedding dimension three, $k_{a,b}$ is a better bound.
\end{rem}

\end{document}